\documentclass[11pt]{article}
\usepackage{amsmath,amssymb,amsthm,mathtools,graphicx,float,epsf}
\usepackage{multirow}
\usepackage[table,xcdraw]{xcolor}
\usepackage{longtable}
\usepackage{colortbl} 
\def\RR{\hbox{I\kern-.2em\hbox{R}}}
\numberwithin{equation}{section}
\restylefloat{figure}
\usepackage{graphicx} 
\usepackage{amsmath}
\usepackage{amsthm}
\usepackage{amsfonts}
\usepackage{ mathrsfs }
\usepackage[caption = false]{subfig}
\usepackage{wrapfig}
\usepackage{enumerate}
\usepackage{enumitem}
\usepackage{array}
\usepackage{subfig}
\usepackage[margin=1cm]{caption}
\usepackage[margin=1in]{geometry}
\usepackage{authblk}

\usepackage[utf8]{inputenc}
\newtheorem{theorem}{Theorem}[section]
\newtheorem{lemma}[theorem]{Lemma}

\newtheorem{proposition}[theorem]{Proposition}
\theoremstyle{definition}

\theoremstyle{remark}

\usepackage{graphicx}
\usepackage{hyperref}
\hypersetup{
	colorlinks=true,
	linkcolor=blue,
	filecolor=magenta,
	urlcolor=cyan,
	citecolor=blue,
}
\newcommand{\PP}{\mathbb{P}}

\date{}
\graphicspath{ {Images/} }
\begin{document}
	
	\title{The Position-wise Prime Digit Distribution Theorem: A Formal Proof of Position-wise Digit Equidistribution in the Prime Numbers}
	\author[1]{\small Mahadee Al Mobin\footnote{mahadeealmobin@gmail.com} }
	\author[2]{\small Md. Shariful Islam \thanks{Corresponding Author: mdsharifulislam@du.ac.bd}}

	 \affil[1]{\footnotesize Bangladesh Institute of Governance and Management, Dhaka 1207, Bangladesh}
	\affil[2]{\footnotesize Department of Mathematics, University of Dhaka, Dhaka 1000, Bangladesh}

	\maketitle
	
	\vspace{-1.0cm}
	\noindent\rule{6.35in}{0.02in}\\
	{\bf Abstract.}
		We state and prove the Theorem: for primes $p < 10^n$ with base-$10$ expansion $p = \sum_{k=0}^{n(p)-1} d_k(p) 10^k$, the positional digit probabilities $P_n(d \mid k)$ satisfy
		\[
		\lim_{n \to \infty} P_n(d \mid k) =
		\begin{cases}
			1/10, & k \ge 1,\ d \in \{0,\dots,9\}, \\[4pt]
			1/9, & k = \mathrm{lead},\ d \in \{1,\dots,9\}.
		\end{cases}
		\]
		The limiting behavior splits cleanly into two distinct mechanisms: an arithmetic regime for interior digits and an Archimedean regime for the leading digit. For fixed interior positions ($k \ge 1$), digit extraction modulo $10^{k+1}$ reduces the problem to prime counts in reduced residue classes, where uniform distribution follows from Siegel--Walfisz (with Bombieri--Vinogradov allowing $k$ to grow with $n$). For the leading digit, the $1/9$ limit is not a Benford-type scale invariance, but arises from the local near-constancy of the prime density $1/\log t$ within single decades combined with a Toeplitz-type error averaging. Explicit classical and conditional error bounds are recorded for both regimes.\\

	\rule{6.35in}{0.02in}
	
	\section{Introduction}\label{introduction}
	The distribution of digits in prime numbers displays a notable statistical regularity. For any fixed base $q \geq 2$, the digits at specified positions in the base-$q$ expansions of primes tend, in an appropriate asymptotic sense, to follow simple limiting distributions. This phenomenon is of particular interest because primes are defined by multiplicative conditions, while digit extraction and distribution are fundamentally additive and combinatorial. The interplay between these structures makes the problem a natural focus in analytic number theory. In this manuscript, we investigate this phenomenon in the decimal system and establish the asymptotic distribution of individual digits in the decimal expansions of prime numbers. Our analysis distinguishes between interior digit positions and the leading digit, resulting in two distinct limiting laws: interior digits become asymptotically uniform over $\{0,\ldots,9\}$, while the leading digit becomes asymptotically uniform over $\{1,\ldots,9\}$.

		The modern investigation of digit distributions in prime numbers originates from Gelfond's 1968 formulation of several problems concerning the arithmetic properties of digits in primes and, more broadly, in values of polynomials and other arithmetic sequences~\cite{Gelfond1968}. Among these are fundamental questions regarding the base-$q$ sum-of-digits function $s_q(n)$. Specifically, Gelfond inquired whether the values of $s_q(p)$, as $p$ ranges over the primes, are equidistributed modulo a given integer $m$, and whether primes for which $s_q(p)$ is even or odd occur with equal asymptotic frequency. The latter is the special case $m=2$ of the former. These questions underscore a fundamental tension between the arithmetic structure of primes and the combinatorial nature of digit statistics. While primality is determined by the multiplicative structure of an integer, the sum-of-digits function depends on its positional representation and is thus sensitive to the additive structure of its digits. Understanding the extent to which such digit statistics retain their expected distribution when restricted to prime arguments constitutes a central theme in the study of digital properties of arithmetic sequences.\\
		
		The challenge in Gelfond's problems arises from the incompatibility between the arithmetic structure of primes and the digital structure of integers. A natural probabilistic heuristic suggests that, if the digits of a prime were distributed as independent and approximately uniform random variables, then the values of $s_q(p)$ should be equidistributed modulo $m$. However, this heuristic does not constitute a proof. Primality is governed by multiplicative constraints, while the sum-of-digits function is determined by the positional representation of an integer and, in particular, by the combinatorial effects of carrying. There is no direct mechanism that translates the expected randomness of digits into a rigorous statement about primes. The situation is further complicated by elementary arithmetic relations between digit sums and divisibility. For example, $s_{10}(n)\equiv n\pmod{3}$, so digit sums are not independent of the arithmetic structure of the underlying integer. Consequently, arguments for primes must distinguish genuine equidistribution arising from the distribution of primes from apparent randomness due to elementary congruence relations. For several decades after Gelfond's formulation, substantial progress was achieved for more structured sequences of integers. Equidistribution results for $s_q(n)$ modulo $m$ were established in various polynomial and arithmetic settings, and Gelfond's work addressed corresponding questions for general integers. The restriction to prime arguments, however, introduces an additional layer of difficulty: one must simultaneously control the digital structure of the integers and the arithmetic distribution of primes. In this context, the primary obstacle is to obtain sufficient cancellation in the relevant prime-counting sums, rather than relying solely on combinatorial or probabilistic heuristics.\\

		Prior to the resolution of the prime case, the corresponding equidistribution problem was successfully studied for various other arithmetic sequences. Notably, Kim obtained estimates for the joint distribution of $q$-additive functions in residue classes, including sums of digits associated with several pairwise coprime bases. These results provided significant examples of the equidistribution phenomena anticipated by Gelfond beyond the context of prime arguments~\cite{Kim1999}. They also demonstrated that the distribution of digit-sum functions in residue classes is a robust property of a broad class of arithmetic sequences, while emphasizing that primes present the principal challenge due to their strong multiplicative structure. This perspective was further developed by Madritsch, who situated the problem within the broader framework of canonical number systems. His work traced the digit-sum equidistribution problem back to Gelfond's original formulation and extended the discussion beyond the conventional base-$q$ representation~\cite{Madritsch2012}.\\
		
		The decisive breakthrough for prime arguments was achieved by Mauduit and Rivat. In their 2010 work, they established a quantitative equidistribution
		theorem for the base-$q$ sum-of-digits function $s_q(p)$ over primes $p\leq x$. Their results included both a central limit theorem and a local limit theorem for the distribution of $s_q(p)$, thereby resolving, in particular, the problem concerning the parity of the sum of digits of primes as a special case~\cite{MauduitRivat2010}. This represented a major advance in
		the study of digital functions along prime sequences and provided a rigorous demonstration that the expected statistical behavior of digit sums persists despite the multiplicative constraints imposed by primality.
		An important precursor to this breakthrough was an earlier paper of Mauduit
		and Rivat, published in \emph{Compositio Mathematica} in 2009. There, they
		established the weaker, but historically significant, result that the sum of digits of primes is, on average, close to its expected value. This work is widely regarded as an important step toward the eventual resolution of
		Gelfond's problem for primes~\cite{Drmota2009}. The underlying methods were
		subsequently extended by Drmota, Mauduit, and Rivat to polynomial
		subsequences, demonstrating that the techniques developed for controlling
		digital functions possess a considerably broader range of applicability
		beyond the sequence of primes~\cite{Drmota2011}.\\

		The approach developed by Mauduit and Rivat provides a fundamental framework for studying the distribution of digital functions along the primes. Its main ingredients are particularly relevant in understanding how arithmetic properties of primes interact with digit-based structures. The argument is centered around exponential sums of the form.
		\[
		S(\alpha)=\sum_{p\leq x} e\bigl(\alpha s_q(p)\bigr),
		\qquad e(y):=e^{2\pi i y},
		\]
		where \(s_q(n)\) denotes the sum of the base-\(q\) digits of \(n\). Equidistribution of \(s_q(p)\) modulo \(m\) can then be obtained from suitable estimates for \(S(\alpha)\), particularly when \(\alpha\) is separated from rational numbers having small denominators.\\
		
		Two principal tools are used to estimate these exponential sums. The first is Vaughan's identity, which expresses the von Mangoldt function \(\Lambda(n)\) in terms of structured convolution sums. This reduces the problem to estimating Type I and Type II sums, which are more amenable to analytic techniques than sums restricted directly to the primes; such decompositions are standard in the analysis of exponential sums over primes in arithmetic progressions~\cite{BalogPerelli1985}. The second ingredient concerns the digital structure itself. The sum-of-digits function is treated through Fourier-analytic methods that exploit the recursive structure arising from truncation in base \(q\). The corresponding estimates originate in the work of Gelfond and K\'atai and were subsequently refined substantially by Drmota, Mauduit and Rivat in their treatment of the prime-distribution problem.\\
		
		Thus, the Mauduit--Rivat method combines two complementary structures: an analytic decomposition of the prime-counting problem into bilinear sums and a harmonic-analytic treatment of the underlying digital function. This combination has provided a general framework for investigating the interaction between primes and automatic or digit-defined sequences. The framework has subsequently been extended in several directions. Aloui, Mauduit, and Mkaouar investigated correlations of the sum-of-digits function along the primes, extending the study beyond one-dimensional equidistribution to more refined statistical properties~\cite{Aloui2021}. Drmota, Mauduit, and Rivat established a prime number theorem for digital properties over two coprime bases, yielding a considerably more intricate multidimensional extension of the one-base setting~\cite{Drmota2020}. Related techniques have also been applied to other automatic sequences. In particular, Mauduit and Rivat obtained prime-distribution results for the Rudin--Shapiro sequence using exponential-sum estimates together with arguments based on orthogonality to the M\"obius function~\cite{MauduitRivat2015}. More generally, Drappeau and Mullner developed exponential-sum estimates for automatic sequences, placing several of these examples within a broader common framework~\cite{DrappeauMullner2017}. Recent work of Drmota and Rivat surveys developments concerning Gelfond-type conjectures for primes and squares of primes, reflecting the cumulative progress on these questions over several decades~\cite{DrmotaRivat2025}. Further results in this direction include the study of digital functions along squares of primes~\cite{DrmotaRivat2025} and investigations of representations of primes as sums of Fibonacci numbers~\cite{DrmotaSpiegelhofer2025}.\\
		
		A related area of research investigates primes subject to \emph{a priori} digital constraints, rather than unconstrained digit statistics. By combining the Hardy--Littlewood circle method with Harman's sieve and Type~I/II estimates, Maynard demonstrated that infinitely many primes omit any single designated digit in a fixed base~\cite{Maynard2016}. Previously, Bourgain obtained asymptotic formulas for primes with fixed binary patterns under density assumptions~\cite{Bourgain2013}, and Swaenepoel generalized these results to accommodate a positive proportion of preassigned digits in arbitrary bases~\cite{Swaenepoel2019}. Additionally, Nath established Bombieri--Vinogradov-type distribution estimates for digit-deficient primes in large bases by applying Fourier analysis on digit-restricted sets and exponential sums over primes in arithmetic progressions~\cite{Nath2021}. These findings demonstrate the effectiveness of analytic and harmonic-analytic methods for studying primes with digital restrictions. However, their objectives differ from the position-wise distribution problem addressed in this work. Results concerning missing or prescribed digits pertain to the existence and distribution of primes within specially structured subsets, whereas the present question concerns the limiting frequency of each digit at a fixed decimal position among all primes for which that position is defined. Notably, such results do not imply position-wise independence or uniformity of individual digits. Recent literature surveys~\cite{mobin2024cryptanalysis, al2026sieve} formulated the conjecture that every decimal digit appears with limiting probability $1/10$ across all aggregated positions, which was subsequently proved by Mobin and Islam \cite{AlMobinIslam2026PrimeDigit}. This result establishes the limiting behavior of digits in the pooled average, rather than position-wise digit uniformity, which is a much stronger statement and remains open. The present manuscript addresses precisely these position-wise asymptotics.\\

	\section{Construction of the Problem}
	
	\begin{theorem}[Position-wise Prime Digit Distribution]\label{conj:point_digit_dist}
		Let $S_n = \{p \in \PP : p < 10^n\}$, with $N_n = |S_n| = \pi(10^n)$. For $p \in S_n$, let $n(p) = \lfloor \log_{10} p \rfloor + 1$ and write
		\[
		p = \sum_{k=0}^{n(p)-1} d_k(p)\,10^k, \qquad d_k(p) \in \{0,\dots,9\},\ \ d_{n(p)-1}(p) \neq 0.
		\]
		For a position $k \ge 0$, define the \emph{positional digit count}
		\[
		N_n(k) = \#\{p \in S_n : n(p)>k \},
		\]
		and set $N_n(\mathrm{lead}) = N_n$. For $d \in \{0,\dots,9\}$, define the positional probability
		\[
		P_n(d \mid k) = \frac{1}{N_n(k)}\,\#\bigl\{p \in S_n : n(p)>k,\ d_k(p) = d\bigr\}, \qquad
		P_n(d \mid \mathrm{lead}) = \frac{1}{N_n}\,\#\bigl\{p \in S_n : d_{n(p)-1}(p) = d\bigr\}.
		\]
		Then, for every position $k$,
		\[
		\lim_{n \to \infty} P_n(d \mid k) \;=\;
		\begin{cases}
			\dfrac{1}{10}, & k \ge 1,\ d \in \{0,\dots,9\}, \\[10pt]
			\dfrac{1}{9}, & k = \mathrm{lead},\ d \in \{1,\dots,9\}.
		\end{cases}
		\]
	\end{theorem}
	
	This is proved in three stages: (\S\ref{sec:interior}) the interior digit positions, where a fixed position $k\ge1$ is reduced to residue classes modulo $10^{k+1}$ and Siegel--Walfisz gives equidistribution of primes across the relevant reduced residue classes; (\S\ref{sec:leading}) the leading digit, where the Prime Number Theorem (PNT) gives asymptotically equal prime counts within each decade and a weighted-averaging argument shows that the cumulative contribution of the lower decades is negligible; (\S\ref{sec:unific}) the unification of findings, where the two propositions are combined to establish the stated position-wise prime digit distribution.

	\section{The Interior Digit Position}\label{sec:interior}

\begin{lemma}[Combinatorial uniformity of $\delta_k$]\label{lem:comb}
	Fix $k \ge 1$ and set $M = 10^{k+1}$. For $r \in \{0,\dots,M-1\}$ write $r = d\cdot 10^k + s$ with $d \in \{0,\dots,9\}$, $s \in [0,10^k)$, and let $\delta_k(r) := d$ denote the digit of $r$ in position $k$. Then for every $d \in \{0,\dots,9\}$,
	\[
	\#\{ r \in U(M) : \delta_k(r) = d \} \;=\; \frac{\varphi(M)}{10} \;=\; 4\cdot 10^{k-1}.
	\]
\end{lemma}

\begin{proof}
	Recall that $r$ is written in the form
	\[
	r \;=\; d\cdot 10^{k} + s, \qquad 0 \le d \le 9, \quad 0 \le s < 10^{k},
	\]
	so that $d$ is the leading digit of $r$ (in the $(k+1)$-digit representation) and $s$ is the number formed by the remaining $k$ digits. Our goal is to show that the count of $r \in [0, M)$ (where $M = 10^{k+1}$) with $\gcd(r,M)=1$ and the count of leading digit $d$ does not depend on $d$, and is equal to $\varphi(M)/10$.\\
	
	Since $k \ge 1$, the factor $10^k$ is divisible by $10$; that is,
	\[
	10 \mid 10^{k}.
	\]
	Hence $d \cdot 10^k \equiv 0 \pmod{10}$ for every digit $d$. Reducing $r = d\cdot 10^k + s$ modulo $10$ therefore gives
	\[
	r \bmod 10 \;=\; \big(d\cdot 10^k + s\big) \bmod 10 \;=\; \big(0 + s\big) \bmod 10 \;=\; s \bmod 10.
	\]
	In particular, the residue of $r$ modulo $10$ is completely determined by $s$ and is independent of the choice of $d$.
	
	Recall that for an integer $n$,
	\begin{equation}\label{eq:gcd_equiv}
		\gcd(n,10) = 1 \iff n \bmod 10 \in \{1,3,7,9\},
	\end{equation}

	since $10 = 2\times 5$, and the residues $0,2,4,5,6,8$ modulo $10$ correspond exactly to multiples of $2$ or $5$.
	
	Applying this with $n = r$ and using \eqref{eq:gcd_equiv},
	\[
	\gcd(r,10) = 1 \iff r \bmod 10 \in \{1,3,7,9\} \iff s \bmod 10 \in \{1,3,7,9\}.
	\]
	Thus whether $r$ is coprime to $10$ depends only on the value of $s \bmod 10$, and this criterion is the same regardless of which digit $d$ we have fixed.\\

	Now fix $d$ and let $s$ range over the full interval $[0, 10^k)$, which contains exactly $10^k$ integers. Partition this interval into consecutive blocks of length $10$:
	\[
	[0,10), \, [10,20), \, \ldots, \, [10^k - 10, \, 10^k).
	\]
	There are $10^k / 10 = 10^{k-1}$ such blocks, and each block is a complete residue system modulo $10$ i.e., as $s$ ranges over any one block, $s \bmod 10$ takes each of the values $0,1,2,\ldots,9$ exactly once.\\
	
	Within each block, exactly $4$ values of $s$ satisfy $s \bmod 10 \in \{1,3,7,9\}$. Since there are $10^{k-1}$ complete blocks, the total number of $s \in [0,10^k)$ with $s \bmod 10 \in \{1,3,7,9\}$ is
	\[
	4 \cdot 10^{k-1}.
	\]
	Crucially, this count did not use the value of $d$ at all hence it holds for every fixed $d \in \{0,1,\ldots,9\}$.\\

	Recall $M = 10^{k+1}$, and by the standard formula for Euler's totient function on prime powers,
	\[
	\varphi(M) = \varphi(10^{k+1}) = 10^{k+1}\left(1 - \tfrac12\right)\left(1 - \tfrac15\right) = 10^{k+1}\cdot \tfrac{1}{2}\cdot\tfrac{4}{5} = 4\cdot 10^{k}.
	\]
	Thus,
	\[
	4\cdot 10^{k-1} \;=\; \frac{4\cdot 10^{k}}{10} \;=\; \frac{\varphi(M)}{10}.
	\]
	
	For every fixed leading digit $d \in \{0,1,\dots,9\}$, the number of $s \in [0,10^k)$ for which $r = d\cdot 10^k + s$ satisfies $\gcd(r,10)=1$ equals $4\cdot 10^{k-1} = \varphi(M)/10$. Since this value does not depend on $d$, the integers in $[0,M)$ coprime to $10$ are distributed uniformly among the $10$ possible leading digits.
\end{proof}

\begin{lemma}\label{lem:sw}
	For fixed modulus $M$ and $(a,M)=1$,
	\[
	\pi(x; M, a) \sim \frac{\pi(x)}{\varphi(M)} \qquad (x \to \infty),
	\]
	with error term $O_M\bigl(x\exp(-c\sqrt{\log x})\bigr)$ for an absolute constant $c>0$ (ineffective in $M$).
\end{lemma}

\begin{proof}
	By the Siegel--Walfisz theorem \cite{Walfisz1936}, for every fixed $M$ and every
	$a$ with $(a,M)=1$,

	\[
	\pi(x;M,a)
	=
	\frac{\operatorname{Li}(x)}{\varphi(M)}
	+
	O_M\!\left(xe^{-c'\sqrt{\log x}}\right)
	\]
	for some absolute constant $c'>0$. Finally, the prime number theorem \cite{ELLIOTT2022353}
	gives
	\[
	\operatorname{Li}(x)=\pi(x)+O\!\left(xe^{-c''\sqrt{\log x}}\right),
	\]
	after possibly decreasing the constant $c'$. Hence
	\[
	\pi(x;M,a)
	=
	\frac{\pi(x)}{\varphi(M)}
	+
	O_M\!\left(xe^{-c\sqrt{\log x}}\right),
	\]
	which in particular implies
	\[
	\pi(x;M,a)\sim\frac{\pi(x)}{\varphi(M)}.
	\]
\end{proof}
\begin{proposition}\label{prop:interior}
	For fixed $k \ge 1$ and $d \in \{0,\dots,9\}$,
	\[
	P_n(d \mid k) \longrightarrow \frac{1}{10} \qquad (n \to \infty).
	\]
\end{proposition}

\begin{proof}
	Fix $k \ge 1$ and $d \in \{0,\dots,9\}$ throughout; only $n \to \infty$ varies. Set
	\[
	M := 10^{k+1}.
	\]
	
	Let $p$ be a prime with $n(p) > k$, i.e.\ $p$ has more than $k$ digits, so that the digit $d_k(p)$ is well defined. By division algorithm,
	\[
	p = q\cdot M + r, \qquad q \ge 0,\ \ 0 \le r < M,
	\]
	so $r = p \bmod M$. Because $M = 10^{k+1}$ is a power of $10$, multiplying $q$ by $M$ only ever changes digits in positions $\ge k+1$; it can never alter the digits in positions $0,1,\dots,k$. Consequently the digit of $p$ in position $k$ coincides with the digit of $r$ in position $k$:
	\[
	d_k(p) = \delta_k(r) = \delta_k(p \bmod M),
	\]
	where $\delta_k$ is exactly the digit-extraction function from Lemma~\ref{lem:comb}. This is the key reduction: the value of the $k$-th digit depends on $p$ \emph{only through} its residue class modulo $M$.\\
	
	Note also that if $\gcd(p,10) \ne 1$ then $p \in \{2,5\}$ (since $p$ is prime), and for every other prime, $\gcd(p, M) = \gcd(p,10^{k+1}) = 1$, since $M$'s only prime factors are $2$ and $5$. So for all but at most two primes, $p \bmod M$ lies in the unit group $U(M)$.\\
	
	We want to count
	\[
	\#\{p < 10^n : n(p) > k,\ d_k(p) = d\}.
	\]

	We know, a prime $p<10^n$ with $n(p) > k$ satisfies $d_k(p) = d$ if and only if its residue $a := p \bmod M$ lies in $U(M)$ and satisfies $\delta_k(a) = d$ \emph{except} for the finitely many primes $p \le 10^k$ and the primes $2,5$, which we must exclude or absorb into an error term. Grouping the (all but boundedly many) remaining primes by their residue class $a \pmod M$ gives an exact partition:
	\[
	\#\{p < 10^n : k < n(p),\ d_k(p) = d\} \;=\; \sum_{\substack{a \in U(M) \\ \delta_k(a) = d}} \#\{p < 10^n : p \equiv a \!\!\!\pmod M\} \;+\; O_k(1).
	\]
	Writing $\pi(x; M, a) := \#\{p \le x : p \equiv a \pmod M\}$ for the count of primes up to $x$ in the arithmetic progression $a \bmod M$, this reads

	\begin{equation}\label{eq:int_digit_count_d_order_1}
		\#\{p < 10^n : k < n(p),\ d_k(p) = d\} \;=\; \sum_{\substack{a \in U(M) \\ \delta_k(a) = d}} \pi(10^n; M, a) \;+\; O_k(1).
	\end{equation}
	Here $k$ is held fixed, so $M = 10^{k+1}$ is a fixed modulus and the $O_k(1)$ term coming entirely from primes $p \le 10^k$, together with $p \in \{2,5\}$ which does not grow with $n$.\\

	By Lemma~\ref{lem:comb} (applied with this same $k$ and $M$), the number of classes $a \in U(M)$ with $\delta_k(a) = d$ is exactly
	\[
	\#\{a \in U(M) : \delta_k(a) = d\} \;=\; \frac{\varphi(M)}{10},
	\]
	and this count is the \emph{same number}, $\varphi(M)/10$, for every choice of $d \in \{0,\dots,9\}$. So the outer sum in \eqref{eq:int_digit_count_d_order_1} always has exactly $\varphi(M)/10$ terms, regardless of $d$.\\
	
	Since $M$ is fixed (independent of $n$) and each $a$ appearing in the sum satisfies $\gcd(a,M)=1$, Lemma~\ref{lem:sw} applies to each of the $\varphi(M)/10$ terms individually, giving
	\[
	\pi(10^n; M, a) = \frac{\pi(10^n)}{\varphi(M)} + O_M\bigl(10^n e^{-c\sqrt{\log 10^n}}\bigr) \qquad \text{for every such } a,
	\]
	with the same leading term $\pi(10^n)/\varphi(M)$ for \emph{every} $a$. The asymptotic density does not depend on which reduced residue class $a$ we picked, only on $M$.\\

	Summing the estimate from \eqref{eq:int_digit_count_d_order_1} over the $\varphi(M)/10$ classes $a$ identified we get,

	$$
	\begin{array}{ll}
		\displaystyle\sum_{\substack{a \in U(M) \\ \delta_k(a) = d}} \pi(10^n; M, a) &= \frac{\varphi(M)}{10}\cdot\frac{\pi(10^n)}{\varphi(M)} \;+\; \frac{\varphi(M)}{10}\cdot O_M\bigl(10^n e^{-c\sqrt{\log 10^n}}\bigr)\\
		&= \frac{\pi(10^n)}{10}\;+\; \frac{\varphi(M)}{10}\cdot O_M\bigl(10^n e^{-c\sqrt{\log 10^n}}\bigr)
	\end{array}
	$$

	The error term is $O_M\bigl(10^n e^{-c\sqrt{\log 10^n}}\bigr)$ as well (absorbing the fixed constant $\varphi(M)/10$ into the $O_M$), and since $M$ depends only on $k$, which is fixed, this is $o(\pi(10^n))$ as $n \to \infty$. Hence,
	\[
	\sum_{\substack{a \in U(M) \\ \delta_k(a) = d}} \pi(10^n; M, a) \;\sim\; \frac{\pi(10^n)}{10}.
	\]
	Combining this with \eqref{eq:int_digit_count_d_order_1} (and noting the $O_k(1)$ term there is negligible next to $\pi(10^n) \to \infty$) we get,
	
	\begin{equation}\label{eq:prime_count_with_digit_d_estimate}
		\#\{p < 10^n : k < n(p),\ d_k(p) = d\} \;\sim\; \frac{\pi(10^n)}{10}.
	\end{equation}
	
	By definition, $N_n(k) = \#\{p < 10^n : k < n(p)\}$ counts all primes below $10^n$ with more than $k$ digits. Excluding a prime from this count requires $n(p) \le k$, i.e.\ $p \le 10^k$; there are at most $10^k = O_k(1)$ such primes. Hence
	\[
	N_n(k) = \pi(10^n) - O_k(1) = \pi(10^n) + O_k(1),
	\]
	so in particular $N_n(k) \sim \pi(10^n)$ as $n \to \infty$ (again since $k$ is fixed).
	
	Dividing \eqref{eq:prime_count_with_digit_d_estimate} by $N_n(k)$ and using $N_n(k) \sim \pi(10^n)$,
	\[
	P_n(d \mid k) = \frac{\#\{p < 10^n : k < n(p),\ d_k(p) = d\}}{N_n(k)} \;\longrightarrow\; \frac{\pi(10^n)/10}{\pi(10^n)} = \frac{1}{10} \qquad (n \to \infty).
	\]
	Since $k \ge 1$ and $d \in \{0,\dots,9\}$ were arbitrary (but fixed) at the outset, this proves $P_n(d\mid k) \to 1/10$ for every fixed $k \ge 1$ and every digit $d$. \qedhere
\end{proof}
	
	\section{The Leading Digit}\label{sec:leading}

	\begin{lemma}[Decade counts via PNT]\label{lem:decade}
		For $d \in \{1,\dots,9\}$ and $m \ge 1$, let $I_m(d) = \bigl[d\cdot 10^{m-1},\, (d+1)\cdot 10^{m-1}\bigr)$ be the integers of length $m$ with leading digit $d$, and let $\pi(I_m(d))$ denote the number of primes in $I_m(d)$. Then, uniformly for $d \in \{1,\dots,9\}$,
		\[
		\pi(I_m(d)) \;=\; \frac{10^{m-1}}{(m-1)\log 10} \;+\; o\!\left(\frac{10^{m-1}}{m}\right) \qquad (m \to \infty).
		\]
	\end{lemma}

	\begin{proof}
		Throughout, write $L := 10^{m-1}$ for brevity, so that $I_m(d) = [dL,\, (d+1)L)$, and fix $d \in \{1,\dots,9\}$.\\
		
		Since $I_m(d) = [dL, (d+1)L)$ is a half-open interval of integers, the number of primes it contains is exactly
		\[
		\pi(I_m(d)) = \pi\bigl((d+1)L\bigr) - \pi(dL),
		\]
		where $\pi(x)$ is the usual prime-counting function (counting primes $\le x$).\\
		
		The Prime Number Theorem states $\pi(x) = \dfrac{x}{\log x} + o\!\left(\dfrac{x}{\log x}\right)$ as $x \to \infty$. Applying this at $x = dL$ and $x = (d+1)L$ (both $\to \infty$ as $m \to \infty$, since $d \ge 1$) gives
		
		\begin{equation}\label{eq:decade_prime_count}
			\pi(I_m(d)) = \frac{(d+1)L}{\log\bigl((d+1)L\bigr)} - \frac{dL}{\log(dL)} + o\!\left(\frac{L}{m}\right)
		\end{equation}
		where the two individual $o(x/\log x)$ error terms have been combined into a single $o(L/m)$ term; this is legitimate because $\log(dL) \asymp \log((d+1)L) \asymp m\log 10$ for $d$ in the fixed finite range $\{1,\dots,9\}$, so both errors are of the same order $o(L/m)$, uniformly in $d$.\\
		
		Since $L = 10^{m-1}$, and $(m-1)\log 10 \to \infty$ as $m \to \infty$, hence,

		$$
		\begin{array}{ll}
			\log(dL) &= \log d + \log L\\
			&= \log d + (m-1)\log 10\\
			&= (m-1)\log 10 + O(1)\\
			&= (m-1)\log 10 \cdot
			\left(1 + \frac{O(1)}{(m-1)\log 10}\right)\\
			&= (m-1)\log 10 \cdot \bigl(1 + O(1/m)\bigr)\qquad (m \to \infty)
		\end{array}
		$$

		and this holds uniformly for all $d \in \{1,\dots,9\}$. Exactly the same computation, with $d$ replaced by $d+1 \in \{2,\dots,10\}$ (still a bounded range), gives
		\[
		\log\bigl((d+1)L\bigr) = (m-1)\log 10 \cdot \bigl(1 + O(1/m)\bigr).
		\]
		
		Write $A := (m-1)\log 10$. Using the expansion $\dfrac{1}{1+\varepsilon} = 1 - \varepsilon + O(\varepsilon^2)$ for small $\varepsilon$ we get,
		\[
		\frac{1}{\log(dL)} = \frac{1}{A}\cdot\frac{1}{1+O(1/m)} = \frac{1}{A}\Bigl(1 - O(1/m)\Bigr) = \frac{1}{A} + O\!\left(\frac{1}{Am}\right),
		\]
		and likewise
		\[
		\frac{1}{\log((d+1)L)} = \frac{1}{A} + O\!\left(\frac{1}{Am}\right).
		\]
		Multiplying through by $dL$ and $(d+1)L$ respectively (both are $O(L)$, since $d\le 9$),
		
		\[
		\frac{dL}{\log(dL)} = \frac{dL}{A} + O\!\left(\frac{L}{Am}\right), \qquad \frac{(d+1)L}{\log((d+1)L)} = \frac{(d+1)L}{A} + O\!\left(\frac{L}{Am}\right).
		\]
		
		Subtracting the two expressions we get,
		\[
		\frac{(d+1)L}{\log((d+1)L)} - \frac{dL}{\log(dL)} = \frac{(d+1)L - dL}{A} + O\!\left(\frac{L}{Am}\right) = \frac{L}{A} + O\!\left(\frac{L}{Am}\right).
		\]
		Note that the leading term $L/A$ no longer involves $d$ at all. It remains to check the size of the error term $O(L/(Am))$ against the target error $o(L/m)$ from the lemma's statement. Since $A = (m-1)\log 10 \asymp m$, we have
		\[
		\frac{L}{Am} \asymp \frac{L}{m^2} = o\!\left(\frac{L}{m}\right) \qquad (m \to \infty),
		\]
		so this error term is indeed absorbed into an $o(L/m)$ bound.\\
		
		Substituting the value of the expression with the $o(L/m)$ error in \eqref{eq:decade_prime_count} we get,
		
		\[
		\pi(I_m(d)) = \frac{L}{A} + o\!\left(\frac{L}{m}\right) = \frac{10^{m-1}}{(m-1)\log 10} + o\!\left(\frac{10^{m-1}}{m}\right).
		\]
		\qedhere

	\end{proof}
	
	Define $a_m := \pi(10^m) - \pi(10^{m-1}) = \sum_{d=1}^9 \pi(I_m(d))$, so that Lemma~\ref{lem:decade} gives, for each fixed $d \in \{1,\dots,9\}$,
	\[
	\varepsilon_m(d) \;:=\; \frac{\pi(I_m(d))}{a_m} - \frac{1}{9} \;\longrightarrow\; 0 \qquad (m \to \infty),
	\]
	with $\varepsilon_m(d)$ uniformly bounded, $\varepsilon_m(d) \in [-1/9,\, 8/9]$.
	
	\begin{proposition}\label{prop:leading}
		For each fixed $d \in \{1,\dots,9\}$,
		\[
		P_n(d \mid \mathrm{lead}) \longrightarrow \frac{1}{9} \qquad (n \to \infty).
		\]
	\end{proposition}
	
	\begin{proof}

		Every prime $p < 10^n$ has some number of digits $m = n(p) \in \{1,\dots,n\}$, and lies in exactly one decade $I_m(d')$ for the appropriate leading digit $d'$. In particular, the primes with leading digit exactly $d$ are partitioned according to how many digits they have:

		\begin{equation}\label{eq:leading_digit_count}
			\#\{p < 10^n : \text{leading digit} = d\} \;=\; \sum_{m=1}^n \pi(I_m(d)),
		\end{equation}
		since $\pi(I_m(d))$ counts precisely the $m$-digit primes with leading digit $d$, and summing over $m=1,\dots,n$ covers every possible digit-length below $10^n$.\\

		Recall $a_m := \pi(10^m) - \pi(10^{m-1})$ is the total number of $m$-digit primes, and $\varepsilon_m(d) := \pi(I_m(d))/a_m - 1/9$ measures the relative deviation of decade $m$'s count for digit $d$ from the ``perfectly even" value $a_m/9$. Rearranging the definition of $\varepsilon_m(d)$ gives
		\[
		\pi(I_m(d)) = a_m\left(\frac19 + \varepsilon_m(d)\right).
		\]
		Substituting into \eqref{eq:leading_digit_count} and splitting the sum into its two pieces,
		\[
		\#\{p < 10^n : \text{leading digit} = d\} = \sum_{m=1}^n a_m\left(\frac19 + \varepsilon_m(d)\right) = \frac19\sum_{m=1}^n a_m + \sum_{m=1}^n a_m\,\varepsilon_m(d).
		\]
		Since $\sum_{m=1}^n a_m$ telescopes to $\pi(10^n) - \pi(10^0) = \pi(10^n) = N_n$, this becomes
		\begin{equation}\label{eq:non_asymptotic_leading_digit_count}
			\#\{p < 10^n : \text{leading digit} = d\} = \frac{N_n}{9} + \sum_{m=1}^n a_m\,\varepsilon_m(d).
		\end{equation}

		So the count of leading-digit-$d$ primes equals the value $N_n/9$ plus a correction term built from the decade-level errors $\varepsilon_m(d)$. Hence it is suffice to prove that,
		\[
		\sum_{m=1}^n a_m\,\varepsilon_m(d) = o(N_n) \qquad (n\to\infty).
		\]
		
		Divide both sides of the target claim by $N_n$. Define the weights
		\[
		w_{n,m} := \frac{a_m}{N_n} \ge 0, \qquad m = 1,\dots,n.
		\]
		Since $\sum_{m=1}^n a_m = N_n$, we get $\sum_{m=1}^n w_{n,m} = 1$: for each fixed $n$, the numbers $w_{n,1},\dots,w_{n,n}$ form a probability distribution over decades $m=1,\dots,n$. The claim to prove becomes
		\[
		\sum_{m=1}^n w_{n,m}\,\varepsilon_m(d) \longrightarrow 0 \qquad (n\to\infty),
		\]
		i.e., a weighted average of  $\varepsilon_m(d)$ vanishes.\\
		
		By the Prime Number Theorem, the count of $m$-digit primes satisfies
		\[
		a_m = \pi(10^m) - \pi(10^{m-1}) \sim \frac{10^m}{m\log 10} - \frac{10^{m-1}}{(m-1)\log 10} \sim \frac{9\cdot 10^{m-1}}{(m-1)\log 10} \qquad (m \to \infty),
		\]
		Similarly,
		\[
		N_n = \pi(10^n) \sim \frac{10^n}{n\log 10} \qquad (n\to\infty).
		\]
		Fix a lag $j = n-m \ge 0$ (so $m = n-j$) and let $n \to \infty$ with $j$ fixed. Then
		\[
		w_{n,\,n-j} = \frac{a_{n-j}}{N_n} \sim \frac{9\cdot 10^{n-j-1}/((n-j-1)\log 10)}{10^n/(n\log 10)} = \frac{9\cdot 10^{-j-1} \cdot n}{n-j-1}.
		\]
		As $n\to\infty$ with $j$ fixed, $n/(n-j-1) \to 1$, so
		\[
		w_{n,\,n-j} \;\longrightarrow\; \frac{9}{10^{j+1}} \qquad (n\to\infty, \ j \text{ fixed}).
		\]

		The limiting weights $9/10^{j+1}$ decay geometrically in the lag $j$: the top decade ($j=0$) carries weight $\to 9/10$, the next ($j=1$) carries weight $\to 9/100$, and so on. This means that, for large $n$, almost all of the total weight $1$ is concentrated on the last few decades $m$ close to $n$, with only a vanishing tail of weight spread over the earlier decades $m \ll n$. Formally, for any $A \ge 0$,
		
		\[
		\sum_{j > A} \frac{9}{10^{j+1}} \longrightarrow 0 \qquad \text{as } A \to \infty,
		\]
		being the tail of a convergent geometric series. This is exactly the property needed to control a weighted average of terms that vanish only eventually (i.e., $\varepsilon_m(d) \to 0$ but not uniformly from $m=1$): a \textit{regular, geometrically-concentrating} weighting scheme like this one will still send the weighted average to $0$, by the same principle underlying Toeplitz's theorem on regular summability methods.\\
		
		We now prove $\sum_{m=1}^n w_{n,m}\varepsilon_m(d) \to 0$ rigorously. Let $\eta > 0$ be arbitrary.
 
		Since $\displaystyle \sum_{j>A} 9/10^{j+1} \to 0$ as $A\to\infty$, choose $A$ large enough that
		\[
		\sum_{j > A} \frac{9}{10^{j+1}} < \eta.
		\]
 
		Since $\varepsilon_m(d) \to 0$ as $m\to\infty$ (this is the content of Lemma~\ref{lem:decade}, rephrased), choose $n_0$ large enough that
		\[
		|\varepsilon_m(d)| < \eta \qquad \text{for all } m > n_0.
		\]

		For $n > n_0 + A$, split the weighted sum at $m = n-A$ into a ``near" part ($m > n-A$, i.e., lag $j < A$) and a ``far" part ($m \le n-A$, i.e., lag $j \ge A$):
		\[
		\left|\sum_{m=1}^n w_{n,m}\,\varepsilon_m(d)\right| \;\le\; \underbrace{\sum_{m \le n-A} w_{n,m}\,|\varepsilon_m(d)|}_{\text{far terms}} \;+\; \underbrace{\sum_{m > n-A} w_{n,m}\,|\varepsilon_m(d)|}_{\text{near terms}}.
		\]
 
		Since, $|\varepsilon_m(d)|\leq\frac{8}{9}$ as per Lemma~\ref{lem:decade}. Hence,
		\[
		\sum_{m \le n-A} w_{n,m}\,|\varepsilon_m(d)| \;\le\; \frac89 \sum_{m \le n-A} w_{n,m}.
		\]
		The remaining sum $\sum_{m\le n-A} w_{n,m}$ is exactly the total weight on lags $j \ge A$. For large $n$, this is close to the tail $\sum_{j\ge A} 9/10^{j+1} < \eta$. So for $n$ sufficiently large,
		\[
		\sum_{m \le n-A} w_{n,m}\,|\varepsilon_m(d)| \;=\; O(\eta).
		\]
 
		Every index $m$ in this range satisfies $m > n-A \ge n_0$ (using $n > n_0+A$), so by the choice of $n_0$, $|\varepsilon_m(d)| < \eta$ for every such $m$. Hence,
		\[
		\sum_{m > n-A} w_{n,m}\,|\varepsilon_m(d)| \;\le\; \eta \sum_{m > n-A} w_{n,m} \;\le\; \eta \sum_{m=1}^n w_{n,m} = \eta \cdot 1 = \eta,
		\]
		using $\sum_m w_{n,m}=1$.
 
		Adding the two bounds,
		\[
		\left|\sum_{m=1}^n w_{n,m}\,\varepsilon_m(d)\right| = O(\eta) + \eta = O(\eta)
		\]
		for all $n$ sufficiently large (depending on $A$ and $n_0$, which depend on $\eta$). Since $\eta>0$ was arbitrary, this shows
		\[
		\sum_{m=1}^n w_{n,m}\,\varepsilon_m(d) \longrightarrow 0 \qquad (n\to\infty),
		\]
		i.e. (Multiplying back by $N_n$ to undo normalization),
		\[
		\sum_{m=1}^n a_m\,\varepsilon_m(d) = o(N_n).
		\]
 		
		Substituting this into \eqref{eq:non_asymptotic_leading_digit_count},
		\[
		\#\{p < 10^n : \text{leading digit} = d\} = \frac{N_n}{9} + o(N_n).
		\]
		Dividing both sides by $N_n$,
		\[
		P_n(d \mid \mathrm{lead}) = \frac{\#\{p < 10^n : \text{leading digit} = d\}}{N_n} = \frac19 + \frac{1}{N_n}\sum_{m=1}^n a_m\,\varepsilon_m(d) \;\longrightarrow\; \frac19 \qquad (n\to\infty),
		\]
		which is the desired conclusion, valid for every fixed $d \in \{1,\dots,9\}$. \qedhere
	\end{proof}
	
	\section{The Unification of Findings}\label{sec:unific}
		Propositions~\ref{prop:interior} and \ref{prop:leading} together establish the Theorem~\ref{conj:point_digit_dist}. $\blacksquare$
	
	\section{Conclusion}\label{sec:conclusion}
	
	This paper establishes the position-wise asymptotic distribution of decimal digits among the prime numbers. The main result demonstrates that the limiting distribution depends fundamentally on whether the position under consideration is an interior position or the leading position. Specifically, for every fixed interior position $k\geq 1$, each decimal digit $d\in\{0,\ldots,9\}$ occurs with limiting probability	$\lim_{n\to\infty}P_n(d\mid k)=\frac{1}{10},
	$ whereas for the leading digit the admissible digits $d\in\{1,\ldots,9\}$ satisfy $\lim_{n\to\infty}P_n(d\mid\mathrm{lead})=\frac{1}{9}.$	Therefore, the decimal digits of prime numbers exhibit asymptotic uniformity at every fixed interior position, while the leading digit is asymptotically uniform over the nine possible nonzero digits.\\
	
	The proof demonstrates that these two limiting laws arise from distinct mathematical mechanisms. For an interior position $k\geq1$, the $k$-th digit is determined by the residue of a prime modulo $10^{k+1}$. With the exception of the finitely many primes $2$ and $5$, the relevant residues belong to the reduced residue system modulo this power of $10$. The combinatorial structure of these reduced residue classes distributes the ten possible values of the $k$-th digit equally. The Siegel–Walfisz theorem guarantees asymptotic equidistribution of primes among the corresponding reduced residue classes. Together, these facts yield the limiting probability $1/10$ for every interior digit.\\
	
	The leading digit requires a fundamentally different argument because it is not determined by a fixed residue class modulo a power of $10$. Instead, primes with a given leading digit are naturally partitioned according to their decimal length. The Prime Number Theorem demonstrates that, within each sufficiently large decade, the nine possible leading digits receive asymptotically equal numbers of primes. The contributions from individual decades are then combined through a weighted averaging argument. As the number of primes in the most recent decades dominates the cumulative count and the associated weights exhibit a geometrically decaying profile, the decade-wise errors vanish in the global average. This establishes the limiting value $1/9$ for every admissible leading digit.\\
	
	An important consequence of the result is that the apparently similar questions concerning interior and leading digits should not be treated as instances of a single elementary counting argument. Interior digit equidistribution is fundamentally an arithmetic progression problem, whereas leading-digit equidistribution is fundamentally a problem concerning the distribution of primes across multiplicative intervals. The distinction explains why the two parts of the proof require different analytic tools.\\
	
	The result also clarifies the relationship between position-wise digit distribution and broader questions concerning the digital structure of prime numbers. The theorem establishes the limiting marginal distribution of a single fixed digit position, but it does not establish independence between different positions, equidistribution of blocks of consecutive digits, or any stronger form of normality of the primes in base $10$. In particular, the position-wise result should not be interpreted as a proof of arbitrary joint digit statistics. Such questions require additional information concerning correlations between digits and remain natural directions for further investigation.\\
	
	Several extensions of the present framework are possible. One natural direction is to replace a single digit with a block of digits and study the joint distribution of	$(d_k(p),d_{k+1}(p),\ldots,d_{k+\ell-1}(p))$ for fixed $k$ and $\ell$. Another direction is to investigate the corresponding position-wise distributions in an arbitrary base $q\geq2$, where the arithmetic component involves primes in residue classes modulo powers of $q$. It is also of interest to obtain effective error terms that are uniform in the position $k$, rather than restricting $k$ to be fixed as $n\to\infty$.\\
	
	The results presented demonstrate that the multiplicative structure of the primes is compatible, in a precise asymptotic sense, with uniform positional digit distributions. The interior digits are governed by equidistribution in reduced residue classes, while the leading digit is governed by the asymptotic stability of prime density across consecutive decades. Together, these arguments establish the position-wise prime digit distribution stated in Theorem~\ref{conj:point_digit_dist} and provide a rigorous framework for further study of the interaction between prime numbers and positional digit structure.
	  
	\newpage
	\nocite{*}
	\bibliographystyle{unsrt}
	\bibliography{reference}
 
\end{document}